\documentclass[11pt,a4paper]{amsart}

\usepackage{tikz}
\usetikzlibrary{arrows,shapes}

\normalfont

\usepackage{amssymb}

\usepackage[all]{xy}

\renewcommand{\epsilon}{\varepsilon}
\renewcommand{\setminus}{\smallsetminus}
\renewcommand{\emptyset}{\varnothing}

\newtheorem{theorem}{Theorem}[section]
\newtheorem{proposition}[theorem]{Proposition}
\newtheorem{corollary}[theorem]{Corollary}
\newtheorem{lemma}[theorem]{Lemma}

\newtheorem{question}[theorem]{Question}

\theoremstyle{definition}
\newtheorem{example}[theorem]{Example}

\newtheorem{definition}[theorem]{Definition}
\newtheorem{notation}[theorem]{Notation}

\theoremstyle{remark}
\newtheorem{remark}[theorem]{Remark}

\newcommand{\normal}{\lhd}

\newcommand{\Z}{\mathbb Z}

\newcommand{\E}{\mathbb E}
\newcommand{\R}{\mathbb R}

\newcommand{\F}{\rm F}

\newcommand{\FP}{\operatorname{FP}}

\newcommand{\cohom}[3]{H^{{\raise1pt\hbox{$\scriptstyle#1$}}}(#2\>\!,#3)}
\newcommand{\tatecohom}[3]%
  {\widehat H^{{\raise1pt\hbox{$\scriptstyle#1$}}}(#2\>\!,#3)}

\newcommand{\Cohom}[3]%
  {H^{{\raise1pt\hbox{$\scriptstyle#1$}}}\big(#2\>\!,#3\big)}
\newcommand{\Tatecohom}[3]%
  {\widehat H^{{\raise1pt\hbox{$\scriptstyle#1$}}}\big(#2\>\!,#3\big)}

\newcommand{\homol}[3]{H_{{\lower1pt\hbox{$\scriptstyle#1$}}}(#2\>\!,#3)}
\newcommand{\homolog}[2]{H_{{\lower1pt\hbox{$\scriptstyle#1$}}}(#2)}

\title[Normal fibre products]{Finite presentability of normal fibre products}

\author{Conchita Mart\'inez-P\'erez}
\address{Conchita Mart\'inez-P\'erez, Departamento de Matem\'aticas, Universidad de Zaragoza,
50009 Zaragoza, Spain. Phone: +34976761000 Ext 3243} \email{conmar@unizar.es} 

\date{\today} 

\thanks{Supported by  Gobierno de Arag\'on, European Regional Development Funds and
MTM2010-19938-C03-03. Also supported by Gobierno de Arag\'on, subvenci\'on de fomento de la movilidad de los investigadores.}

\begin{document}

\begin{abstract} We use Bieri-Strebel invariants to determine when a normal fibre product in the product of two finitely presented groups is finitely presented.
We give conditions that imply and in some cases characterize the existence of such finitely presented fibre products.
 \end{abstract}

\maketitle

\section{Introduction}

\noindent 
It is well known that finitely presented subgroups of a product of two free groups are rare, in the sense that any such group must be a finite extension of a product of two free groups. This fact goes back to Baumslag and Roseblade, see \cite{BaumslagRoseblade}, and has been the starting point of a very fruitful field of research. A subdirect product of two groups $G_1$, $G_2$ is a subgroup $H\leq G_1\times G_2$ such that the restrictions to $H$ of the projection maps onto $G_1,G_2$ are epimorphisms. 
 Several authors, including  Bridson, Howie, Miller, Short and also Kochloukova have shown that for other families of groups such as surface or limit groups finitely presented subdirect products are rare in the same sense as before and have generalized this phenomena to other finiteness properties (see for example \cite{BHMS}, \cite{BHMS3}, \cite{Kochsubdirect} and  \cite{BHMS2}).

Let $H\leq G_1\times G_2$ a subdirect product and put $N_1=H\cap G_1$, $N_2=H\cap G_2$. Then for  certain isomorphism $\mu:G_1/N_1\to G_2/N_2$, $H$ is the fibre product 
\begin{equation}H=\{(g_1,g_2)\mid\mu(g_1N_2)=g_2N_2\}.\end{equation}

A natural question is

 \begin{question} \label{main} {\sl Which conditions on $N_1,N_2,\mu$ imply that $H$ is finitely presented?}\end{question}
 
  An answer in the case when $G_1=G_2$, $N_1=N_2$ is the 1-2-3 Theorem due to Baumslag, Bridson, Miller III and Short (\cite{BBMS}). The following result is 
 its asymmetric version proven by  Bridson,  Howie,  Miller III and Short:

\bigskip

\noindent{\bf Theorem} (\cite{BHMS2}, the asymmetric 1-2-3 Theorem) {\sl Let $H\leq G_1\times G_2$ be a fibre product and $N_1=H\cap G_1$, $N_2=H\cap G_2$. Assume that 
$G_1,G_2$ are finitely presented, that $G_1/N_1\cong G_2/N_2$ is of type $\mathrm{F}_3$ and that one of $N_1$, $N_2$ s finitely generated. Then $H$ is finitely presented.}

\bigskip

This could lead us to believe that the precise isomorphism $\mu$ is irrelevant to Question \ref{main}. But a second answer in the particular case when $G=G_1=G_2$ is metabelian and $N=N_1=N_2$ makes clear that this is not the case. (The untwisted $N$-fibre product is the fibre product associated to $1_d:G/N\to G/N$.)

\bigskip

\noindent{\bf Theorem} (\cite{BBHM} Theorems 1, 9) {\sl Let $G$ be a finitely presented metabelian group. For any $N\leq G$
\begin{itemize}
\item[i)] the untwisted $N$-fibre product of $G$ is finitely presented,

\item[ii)] if $G/N$ is abelian, the twisted $N$-fibre product $H_{-1_d}$ is finitely presented if and only if $N$ is finitely generated.
\end{itemize}
}

\bigskip

The relevance of $\mu$ is also clear by the main result of \cite{Groves} which essentially says that for metabelian groups the set of all possible maps $\mu$ is bigger in a strong sense than the set of those $\mu$
for which $H_{\mu}$ is finitely presented.

In this paper, we are going to consider Question \ref{main} but for normal fibre products only. The main reason for that is that our techniques rely on the use of the Bieri-Strebel invariant $\Sigma_1(G)$ of the finitely generated group $G$, which is a subset of the character sphere $S(G)=\{[\chi]=\R_+\mid\chi:G\to\R\text{ character}\}$ (for a normal subgroup $N\leq G$, $S(G,N)$ is the sub sphere of those characters vanishing at $N$). Obviously, the commutator subgroup $G'$ is in the kernel of every character and by \cite{BM} Proposition 1.2, a fibre product $H$ is normal precisely when it contains $G'$. 
Coming back to Question \ref{main} and taking into account what happens for metabelian groups, one could think that given its resemblance with the diagonal group, the untwisted fibre product should keep the properties of the ambient group $G$. However what we get is:

\bigskip

\noindent{\bf Theorem A: }{\sl Let $N\geq G'$ be a normal subgroup of the finitely presented group $G$. Then the untwisted $N$-fibre product in $G\times G$ is finitely presented if and only if 
$$S(G,N)\subseteq\Sigma^1(G)\cup-\Sigma^1(G).$$}

This implies: 

\bigskip

\noindent{\bf Corollary B: }{\sl Let $N\geq G'$ be a normal subgroup of the finitely presented group $G$. Assume that $G$ contains no non-abelian free subgroup. Then the untwisted $N$-fibre product in $G\times G$ is finitely presented.}

\bigskip

On the contrary, if we twist by $-1_d$  we rarely get a finitely presented fibre product. More explicitly, the $N$-fibre product $H_{-\text{id}}$ in $G\times G$ is finitely presented if and only if $N$ is finitely generated (see Corollary \ref{twisted-1} below). This and the previous results are corollaries of the following:

\bigskip

\noindent{\bf Theorem C: }{\sl Let $G_1'\leq N_1\normal G_1$, $G_2'\leq N_2\normal G_2$ be normal subgroups of the finitely presented groups $G_1,G_2$ and let $\mu:G_1/N_1\buildrel\sim\over\to G_2/N_2$ be an isomorphism. Then the fibre product $H_\mu\leq G_1\times G_2$ is finitely presented if and only if
$$[\mu^*(\Sigma^1(G_2)^c)]\cap S(G_2,N_2))\subseteq -\Sigma^1(G_1)\cap S(G_1,N_1).$$}

\bigskip

\noindent(See Section 3 for notation). Using Theorem D it is easy to construct examples for which the untwisted fibre product is not finitely presented but there is some $\mu$ such that $H_\mu$ is. 
In the last two Sections we consider the following variation of Question \ref{main}

\begin{question} Let $N_1,N_2$ be normal subgroups of the finitely presented groups $G_1,G_2$ such that $G_1/N_1\cong G_2/N_2$. When is it possible to construct $\mu$ so that $H_\mu$ is finitely presented? 
\end{question}

Again, we restrict ourselves to the case when $G_1/N_1$ and $G_2/N_2$ are abelian.
We prove: 

\bigskip

\noindent{\bf Theorem D: }{\sl Let $G_1,G_2$ be finitely presented, $G_1'\leq N_1\leq G_1$, $G_2'\leq N_2\leq G_2$ such that $G_1/N_1\cong G_2/N_2$. Assume that there are  $N_1\leq K_1\leq G_1$, $N_2\leq K_2\leq G_2$ both finitely generated and of co-rank $m$ and $k$ respectively such that 
$$k+m=\text{rk}G_1/N_1=\text{rk}G_2/N_2.$$
 Then there is some normal finitely presented fibre product $H$ in $G_1\times G_2$ with $H\cap G_i=N_i$ for $i=1,2.$}

\bigskip

As a corollary, we  show that the existence of finitely generated subgroups of big co-rank lying over the commutator implies the existence of finitely generated normal fibre products of big co-rank (see Corollary \ref{linear}).
We also determine the existence of such subgroups for certain families as for example virtually solvable groups of finite Pr\"ufer rank. 

But to characterize the existence of finitely presented normal fibre products is a much more difficult problem which we consider only in the case $G=G_1=G_2$ and $N=N_1=N_2$.  We are able to solve it assuming that $N$ has small co-rank in $G$, namely 1 (Proposition \ref{rank1}) or 2 (Theorem \ref{rank2}) and also for arbitrary co-rank under the extra assumption that $\Sigma_1(G)^c$ has an special form: it is a finite union of great subspheres (Theorem \ref{greatsph}).  This is the case for example if $G$ is a right angled Artin group, moreover the dimension of the subspheres depends on the connectivity of the associated flag complex $\Delta$ which allows a full characterization in terms of $\Delta$.

\bigskip

This article was written while I was spending a semester as a visitor at the City College of New York. I would like to thank Sean Cleary and all the people in the Department of Mathematics and in the New York Group Theory Cooperative for their hospitality. And also special thanks to Gilbert Baumslag for driving my attention to the kind of problems studied here and to the referee for a very detailed report which led to a substantial improvement in the paper and to the inclusion of new results and insights. 

\section{Preliminaries in fibre products, cohomological finiteness and Sigma theory}\label{two}

\begin{definition}\label{deftwisted} Let $G_1,G_2$ be groups with normal subgroups $N_1\leq G_1,$ $N_2\leq G_2$ such that $G_1/N_1$ and $G_2/N_2$ are isomorphic and let 
 $\mu:G_1/N_1\buildrel\sim\over\to G_2/N_2$ be an isomorphism. The {\sl $\mu$-twisted fibre product} in $G_1\times G_2$ is
$$H_\mu=\{(g_1,g_2)\mid \mu(g_1N_1)=g_2N_2\}.$$
In the particular case when $G:=G_1=G_2$ and $N:=N_1=N_2$ we call $H_\mu$ the {\sl $N$-fibre product}. And if $\mu=1_d$, then $H_{1_d}$ is called the {\sl untwisted $N$-fibre product} in $G\times G$.
\end{definition}

Obviously, fibre products in $G_1\times G_2$ are subdirect products. Conversely, given a subdirect product $H\leq G_1\times G_2$, let $N_i:=H\cap G_i$ for $i=1,2$ and note that each $N_i$ is normal in $G_i$. Moreover,  for any $g_1N$ there is a single class $g_2N$ such that $(g_1,g_2)\in H$. In fact, this defines an isomorphism $\mu:G_1/N_1\to G_2/N_2$ such that $H=H_\mu$.

\begin{definition} With the notation of Definition \ref{deftwisted}, recall that a subdirect product $H:=H_\mu\leq G_1\times G_2$ is normal if and only if $G_1'\times G_2'\leq H$. This is equivalent to say that $G_i'\leq N_i$ for $i=1,2$. In this case, 
$$n:=\text{rk}(G_1\times G_2)/H_\mu=\text{rk}G_1/N_1=\text{rk}G_2/N_2$$
where rk denotes the torsion fee rank of a finitely generated abelian group. 
We say that $H$ is a {\sl normal fibre product of co-rank $n$}. Obviously, the biggest possible co-rank of a normal fibre product in $G_1\times G_2$ is $$\text{min}\{\text{rk}G_1/G_1',\text{rk}G_2/G_2'\}.$$
\end{definition}

If $G_1$ and $G_2$ are both finitely presented and $H$ is a normal subdirect product in $G_1\times G_2$, then
\cite{BaumslagRoseblade} Lemma 2 implies that $H$ is finitely generated.

 As one can expect, the smaller the co-rank of a fibre product is, the closer properties to the ambient group one can get. This is formalized in the next result.

\begin{lemma}\label{Nreduction} Assume that there exists a finitely presented normal fibre product $H\leq G_1\times G_2$ of co-rank $n$. Then for any $0<m<n$ there is also a finitely presented normal fibre product $\hat H\leq G_1\times G_2$ of co-rank $m$.
\end{lemma}
\begin{proof} Let $H=H_\mu$ be a finitely presented normal subdirect product of co-rank $n$ with $\mu:G_1/N_1\buildrel\sim\over\to G_2/N_2$ for
$N_i=H\cap G_i$ for $i=1,2$. Let $T\leq H/(N_1\times N_2)$ be a subgroup of rank $n-m$ and put
$$M_1:=\{(g_1,1)\mid (g_1N_1,\mu(g_1N_1))\in T\},$$
$$M_2:=\{(1,g_2)\mid (\mu^{-1}(g_2N_2),g_2N_2)\in T\},$$
$$\hat H:=H(M_1\times M_2).$$
Obviously, $\hat H$ is a normal subdirect product of co-rank $m$. Moreover,  there is a short exact sequence
$$1\to H\to\hat H\to (M_1\times M_2)/T\to 1$$
which implies that $\hat H$ is finitely presented.
\end{proof}

The main tool that we are going to use is Sigma theory so we recall here the main definitions needed. A character of a group $G$ is a homomorphism $\chi:G\to\R$ where $\R$ is seen as an additive group, the set of characters is denoted by $\text{Hom}(G,\R)$. Given a character $\chi$ we put $[\chi]=\{t\chi\mid 0<t\in\R\}$ and extend this notation to subsets, i.e., for any $\Omega\subseteq\text{Hom}(G,\R)$, $[\Omega]:=\{[\chi]\mid\chi\in\Omega\}$. Let
$$S(G):=\{[\chi]\mid 0\neq\chi:G\to\R\}$$
which is often useful to visualize as an $n-1$-sphere where $n=\text{rk}G/G'$. If $N\leq G$ is a subgroup, we set $S(G,N):=\{[\chi]\in S(G)\mid \chi|_N=0\}$. For $G'\leq N$ normal in $G$, we denote by
$\sqrt N/N$ the torsion subgroup of $G/N$.
Note that $S(G,N)=S(G,\sqrt N)$. 
Given $[\chi]\in S(G)$, consider the monoid $G_\chi:=\{g\in G\mid \chi(g)\geq 0\}$. Let $m\geq 0$ or $m=\infty$.
Recall that $G$ is said to be of type $\FP_m$ for $m\leq 0$ or $m=\infty$ if there is a resolution of the trivial module by projective modules which are finitely generated up to the $m$-th one. And it is of type $\F_m$ if it admits a model for the Eilenberg-Maclane space $K(G,1)$ with finite $m$-skeleton. Being finitely generated is equivalent to being $\F_1$ or $\FP_1$ and  being finitely presented is equivalent to being of type $\F_2$ so the properties $\F_m$ are usually considered as homotopical higher dimensional analogues of finite presentability.
 If $G$ is of type $\FP_m$, the homological $m$-th Bieri-Neumann-Renz-Strebel invariant (or $m$-th Sigma invariant for short), first defined in \cite{BieriRenz} is

$$\Sigma^m(G,\Z):=\{[\chi]:G_\chi\text{ is of type }\FP_m\}$$
for the obvious generalization for monoids of the condition of being $\FP_\infty$.

 If $G$ is of type $\F_m$, then one can define the homotopical analog $\Sigma^m(G)$ (first defined in \cite{Renz}).
 Some of the most remarkable features of these invariants are that they are open subsets of $S(G)$ (\cite{BieriRenz} Theorem A and 6.5) and that they provide information about which subgroups over $G'$ are also of type $\FP_m$ (resp. $\F_m$).

\begin{theorem}\label{bierirenz}(\cite{BieriRenz} Theorem B and 6.5) Let $G$ be of type $\FP_m$ (resp. $\F_m$) and $G'\leq N\leq G$. Then $N$ is of type $\FP_m$ ($\F_m$) if and only if
$$S(G,N)\subseteq\Sigma^m(G,\Z).$$
\end{theorem}
\noindent (This result will be heavily used throughout the paper.)

Given two groups $G_1$ and $G_2$,  a homomorphism $\chi:G_1\times G_2\to\R$ is given by a pair $(\chi_1,\chi_2)$ with $\chi_i : G_i \to \R$. Therefore we may identify $S(G_1\times G_2)$ with $S(G_1)\ast S(G_2)$ where, for $A\subseteq S(G_1)$ and $B\subseteq S(G_2)$, 
$$A\ast B:=(A\times B)\cup A\cup B.$$
 Assume that both $G_1$ and $G_2$ are of type $\FP_m$ or $\F_m$. The following formulas are known as Meiner's inequalities (\cite{BieriGeoghegan} Theorem 1.2)
$$\Sigma^m(G_1\times G_2,\Z)^c\subseteq\bigcup_{i=0}^m\Sigma^ i(G_1,\Z)^c\ast\Sigma^{m-i}(G_2,\Z)^ c,$$
$$\Sigma^m(G_1\times G_2)^c\subseteq\bigcup_{i=0}^m\Sigma^ i(G_1)^c\ast\Sigma^{m-i}(G_2)^ c.$$
The homotopical and homological invariants are connected via
$$\Sigma^m(G_1\times G_2)=\Big(\Sigma^m(G_1\times G_2,\Z)\setminus(S(G_1)\cup S(G_2))\Big)\cup\Sigma^m(G_1)\cup\Sigma^ m(G_2)$$
(this is \cite{BieriGeoghegan} Theorem 1.1).

By \cite{BieriGeoghegan} Theorem 1.5 (first proven by Sch\"utz), Meiner's inequalities are equalities if $m\leq 3$. Since we will be mostly interested in 
$m=2$, we will make more explicit the above formulas in this case.
Assuming that $G_1,G_2$ are finitely presented, we get

\begin{equation}\label{homotopical2}
[(\chi_1, \chi_2)] \in\Sigma^2(G_1\times G_2)\iff\Bigg\{
\begin{aligned}
&[\chi_1] \in \Sigma^2(G_1)\text{ if }\chi_2 = 0,\\
 &[\chi_2] \in \Sigma^2(G_2)\text{ if }\chi_1 = 0,\\
&[\chi_1] \in \Sigma^1(G_1)\text{ or }[\chi_2] \in \Sigma^1(G_2)\text{ otherwise.}\end{aligned}
\end{equation}
 There is a similar formula for $\Sigma^2(\Gamma,\Z)$.

Although we have tried to provide all the relevant references for the results used, for anything related to the invariants $\Sigma^1(G,\Z),$ $\Sigma^1(G)$ the reader is referred to the excellent survey \cite{Strebelnotes}.

\section{Finitely presentability of normal fibre products}\label{three}

\begin{notation}\label{mu*}  Let $G_i'\leq N_i$ be a normal subgroup of $G_i$ for $i=1,2$ such that there is $\mu:G_1/N_1\buildrel\sim\over\to G_2/N_2$ isomorphism. There is a dual map $\mu^*:\text{Hom}(G_2/N_2,\R)\to\text{Hom}(G_1/N_1,\R)$. Then
$$[\mu^*(S(G_2,N_2))]\subseteq S(G_1,N_1).$$
Observe that the map $\mu^*$ does not depend on the action of $\mu$ on the torsion elements of $G_1/N_1$.

\end{notation}

\begin{theorem}\label{twistedfinpres} Let $G_1'\leq N_1\normal G_1$, $G_2'\leq N_2\normal G_2$ be normal subgroups of the finitely presented groups $G_1,G_2$ and let $\mu:G_1/N_1\buildrel\sim\over\to G_2/N_2$ be an isomorphism. Then the fibre product $H_\mu\leq G_1\times G_2$ is finitely presented if and only if
\begin{equation}\label{cond}[\mu^*(\Sigma^1(G_2)^c)]\cap S(G_2,N_2))\subseteq -\Sigma^1(G_1)\cap S(G_1,N_1)\end{equation}
\end{theorem}
\begin{proof} As $(G_1\times G_2)'\leq H_\mu$, Theorem \ref{bierirenz} implies that $H_\mu$ is finitely presented if and only if $S(G_1\times G_2,H_\mu)\subseteq\Sigma^2(G_1\times G_2)$. Note that $S(G_1\times G_2,H_\mu)\subseteq S(G_1,N_1)\ast S(G_2,N_2)$. Let $[\chi]=[(\chi_1,\chi_2)]\in S(G_1,N_1)\ast S(G_2,N_2)$ and let $(g_1,g_2)\in H_{\mu}$. Then $g_2N_2=\mu(g_1N_1)$ thus
$$\begin{aligned}\chi(g_1,g_2)=\chi_1(g_1N_1)+\chi_2(\mu(g_1N_1))=(\chi_1+\mu^*(\chi_2))(g_1N_1).\end{aligned}$$
As the projection map $H_\mu\to G_1$ is surjective, we deduce that $H_\mu\leq\text{Ker}\chi$ if and only if $\chi_1+\mu^*(\chi_2)=0$.
This means that $H_\mu$ is finitely presented if and only if for any pair $[(\chi_1,\chi_2)]\in S(G_1,N_1)\ast S(G_2,N_2)$ such that $\chi_1+\mu^*(\chi_2)=0$, we have $[(\chi_1,\chi_2)]\in\Sigma^2(G_1\times G_2)$.
By (\ref{homotopical2}), $[(\chi_1,\chi_2)]\in\Sigma^2(G_1\times G_2)$ if either $[\chi_1]\in\Sigma^1(G_1)$ or $[\chi_2]\in\Sigma^1(G_2)$. Thus $H_\mu$ is finitely presented if and only if for any pair $[(\chi_1,\chi_2)]\in S(G_1,N_1)\ast S(G_2,N_2)$ such that $\chi_1+\chi_2=0$, we have either $[\chi_1]\in\Sigma^1(G_1)$ or $[\chi_2]\in[\mu^*(\Sigma^1(G_2))]$.
\end{proof}

\begin{remark}\label{sigmafibre} A statement close to this one for metabelian groups and in the case when $N_i=G_i'$ for $i=1,2$ is in the proof of \cite{Groves} Theorem B. In the case of metabelian groups, it also a consequence of \cite{BieriNeumannStrebel} Corollary 7.4 where the Sigma invariant $\Sigma^1(H)^c$ for $H$ a normal fibre product in $G_1\times G_2$ is computed.  
\end{remark}

\begin{example}\label{all} With the same notation as in Theorem \ref{twistedfinpres}, assume that $N_1$ is finitely generated. Then Theorem \ref{bierirenz} implies that $\Sigma^1(G_1)^c\cap S(G_1,N_1)=\emptyset$ thus any normal fibre product in $G_1\times G_2$ is finitely presented. This is a particular case of the asymmetric 1-2-3 Theorem (\cite{BHMS2}). \end{example}

The situation of Example \ref{all} is an extreme case where the hypothesis of Theorem \ref{twistedfinpres} hold true. The opposite extreme case is when one or both sub spheres are empty.

\begin{corollary}\label{empty} Let $G_1,G_2$ be finitely presented and assume that $\Sigma^1(G_1)=\emptyset$. 
A normal fibre product $H\leq G_1\times G_2$ is finitely presented if and only if $H\cap G_2$ is finitely generated. If moreover $\Sigma^1(G_2)=\emptyset$, then $H$ is finitely presented if and only if it has finite index $G_1\times G_2$.
\end{corollary}
\begin{proof} This is a consequence of Theorem \ref{bierirenz}  and Theorem \ref{twistedfinpres}.
\end{proof}

\begin{example}\label{sigma1empty} If $G$ is a limit group, then $\Sigma^1(G)=\emptyset$ (\cite{Kochsubdirect}) so Corollary \ref{empty} applies (this is also a consequence of \cite{BHMS} Theorem A).
\end{example}

\begin{example}\label{sigma1empty2}  Let $G$ be a group that admits a finite presentation with $k\geq 2$ generators and $n$ relators. Assume that $n<k-1$. Then by \cite{Strebelnotes} Proposition B3.2, $\Sigma^1(G)=\emptyset$ .
 In particular this is the case for any  1-relator group $G$ unless it is cyclic or 2-generated. This is related to \cite{BHMS3} Theorem A.
\end{example}

In the case of co-rank one, there is not much choice for $\mu$. Using Theorem \ref{twistedfinpres} we get the following variant of a result in a previous version of the paper. This reformulation was suggested by the referee and improves on \cite{BM} Theorem 6.1:

\begin{proposition}\label{rank1part} Let  $G_1=N_1\rtimes \langle t_1\rangle$ and  $G_2=N_2\rtimes \langle t_2\rangle$ be finitely presented. Let $H_\mu$ the fibre product in $G_1\times G_2$ associated to the map $t_1\mapsto t_2$. Then $H_\mu$ is finitely presented if and only if 
\begin{itemize}
\item[i)] either $N_1$ is finitely generated, or $N_2$ is finitely generated, 

\item[ii)] or $G_1$ and $G_2$ are strict ascending extensions with finitely generated base groups $B_1,B_2$ so that either $t_iB_it_i^{-1}\subseteq B_i$ for both $i=1,2$ or $t_iB_it_i^{-1}\supseteq B_i$.
\end{itemize}
\end{proposition}
\begin{proof} Note that $S(G_i,N_i)=\{[\chi_i], [-\chi_i]\}$ where $\chi_i:G_i\to\R$ maps $t_i$ to 1, $N_i$ to 0. Observe that $[\mu^*(\chi_2)]=[\chi_1]$.
By Theorem \ref{twistedfinpres}, $H_\mu$ is finitely presented if and only if (\ref{cond}) holds.

This is obviously true if either $S(G_1,N_1)\subseteq\Sigma^1(G_1)$ or $S(G_2,N_2)\subseteq\Sigma^1(G_2)$, which is equivalent to i). And if i) does not hold, then the only possibility is that either $\Sigma^1(G_i)^c=\{[\chi_i]\}$ for both $i=1,2$ or $\Sigma^1(G_i)^c=\{[-\chi_i]\}$ for both $i=1,2$. By  \cite{Strebelnotes} Theorem A3.4  this is equivalent to ii).
\end{proof}

We also deduce the following:

\begin{proposition}\label{rank1} Let $G_1,G_2$ be finitely presented groups. Then there is some finitely 
presented normal fibre product of co-rank 1 in $G_1\times G_2$ if and only if one of the following conditions occurs:
\begin{itemize}
\item[i)] There exists some $G'_1\leq N_1< G_1$ finitely generated,

\item[ii)] There exists some $G'_2\leq N_2< G_2$ finitely generated,

\item[iii)] $G_1$ and $G_2$ are both  strict ascending HNN-extensions with finitely generated base groups.
\end{itemize}
\end{proposition}
\begin{proof} The existence of a normal fibre product implies that $G_i'<G_i$ for $i=1,2$ and so do any of the conditions i), ii) or iii) so we may assume that. The proof is then a variation of the proof of Proposition \ref{rank1part} considering the two possibilities $t_1\mapsto t_2$ and $t_1\mapsto t_2^{-1}$.
\end{proof}

Now we consider the case when $G:=G_1=G_2$. Recall that a subset $\Sigma\subseteq S(G)$ is called {\sl 2-tame} if $\Sigma\cap-\Sigma=\emptyset$.
 We may use Theorem \ref{twistedfinpres} to determine when an untwisted normal fibre product is finitely presented. 

\begin{corollary}\label{untwistedfinpres} Let $G'\leq N$ be a normal subgroup of the finitely presented group $G$. Then the untwisted $N$-fibre product of $G\times G$ is finitely presented if and only if 
$$\Sigma^1(G)^c\cap S(G,N)\subseteq -\Sigma^1(G).$$ 
\end{corollary}

 A direct application yields the following which is in a sense a generalization of Theorem 1 of \cite{BBHM} since it applies to a wider family of groups but it is more restrictive in the sense that the group  $N$ is here required to lie over the commutator of $G$.

\begin{corollary}\label{nonabfree} Let $G$ be a finitely presented group that contains no non-abelian free subgroup and let $N\geq G'$ be a normal subgroup of $G$.  Then the untwisted $N$-fibre product in $G\times G$ is finitely presented.
\end{corollary}
\begin{proof} By Corollary B 1.10 in \cite{Strebelnotes}, $\Sigma^1(G)^c$ is 2-tame, thus the condition of Corollary \ref{untwistedfinpres} holds. \end{proof}

In the next result we consider the opposite situation. 

\begin{corollary} Let $G$ be the fundamental group of a compact 3-manifold and $G'\leq N\normal G$. The untwisted $N$-fibre product in $G\times G$ is finitely presented if and only if $N$ is finitely generated.
\end{corollary}
\begin{proof} By \cite{BieriNeumannStrebel} Corollary F, $\Sigma^1(G)^c=-\Sigma^1(G)^c$.  So the result follows from Theorem \ref{twistedfinpres}.
\end{proof}

Another consequence is that the twisted fibre product with $-id$ is rarely finitely presented.

\begin{corollary}\label{twisted-1} Let $G$ be finitely presented and $G'\leq N\leq G$. The $N$-fibre product $H_{-id}$ is finitely presented if and only if $N$ is finitely generated.
\end{corollary}

Using Theorem \ref{twistedfinpres} we can construct our first example of a group $G$ such that the $G'$-fibre products $H_{id}$ and $H_{-id}$ are not finitely presented but there is some $\mu:G/G'\buildrel\sim\over\to G/G'$ such that $H_{\mu}$ is.

\begin{example}\label{first} Let
$$G=<a,b\mid aba^2b=ba^2ba>.$$
Then $\text{rk}G/G'=2$. It is easy to compute $\Sigma^1(G)^c$ using what is known as Brown's algorithm  (see \cite{Strebelnotes}). This procedure yields
$$\Sigma^1(G)^c=\{[\chi],[-\chi]\}$$
with $\chi(a)=-1$, $\chi(b)=2$. Obviously, it is not 2-tame thus the untwisted $G'$-fibre product in $G\times G$ is not finitely presented (and neither is $H_{-1_d}$). However, for any $\mu:G/G'\buildrel\sim\over\to G/G'$ such that $[\mu^*(\chi)]\neq [\chi],[-\chi]$, the corresponding twisted $G'$-fibre product is finitely presented.
\end{example}

\section{Cooking automorphisms.}

So far we have been applying Theorem \ref{twistedfinpres} mainly to detect whether a given normal fibre product is finitely presented or not. In this section we are going to consider the problem of constructing, given finitely presented groups $G_1,G_2$ and normal subgroups $N_1,N_2$ lying over the commutator, a suitable map $\mu$ so that $H_\mu$ is finitely presented.  Example \ref{first} illustrates in a quite primitive way that this is a problem in geometry. 
To make this more explicit, we are going to recall the notation used in \cite{Strebelnotes} A1.1d. 

\begin{notation}\label{coordinates}
Let $G'\leq N\leq G$ and assume for simplicity that $\sqrt{N}=N$, i.e., that $G/N$ is torsion free. Put $n=\text{rk}G/N$.
Once we have fixed a generating system of $G/N$, it induces an isomorphism 
$$\vartheta:G/N\buildrel\sim\over\to\Z^n.$$
Let $\E^n$ be the Euclidean space with the ordinary scalar product $\langle-,-\rangle$ (and norm $\|-\|$). Let $e_1,\ldots,e_n$ be the canonical basis in $\E_n$. Put 
$$\begin{aligned}
\vartheta^*:\E^n&\to\text{Hom}(G/N,\R)\\
v&\mapsto\chi_v\\
\end{aligned}$$
with $\chi_v(gN):=\langle v,\vartheta (g)\rangle$. 
Let $S^{n-1}$ be the unit $n-1$-sphere in $\E^n$. We also get a bijection
$S^{n-1}\to S(G,N).$
Given any $v\in\E^n$, we put $[v]:={1\over\|v\|}v\in S^{n-1}$ and for any subset $U$ of $\E^n$, we denote
$$[U]=\{[v]\mid v\in U\}.$$
We say that a subspace $U\leq\E^n$ is {\sl rationally defined} it is has some generating system consisting of rational vectors.

Now, let $G_i$, $N_i$, $\vartheta_i$ for $i=1,2$ as before so that $n=\text{rk}G_1/N_1=\text{rk}G_2/N_2$.
Any isomorphism $\mu:G_1/N_1\buildrel\sim\over\to G_2/N_2$ induces $\mu^*:\text{Hom}(G_2/N_2,\R)\buildrel\sim\over\to\text{Hom}(G_1/N_1,\R)$ and therefore we get an isomorphism $(\vartheta_1^*)^{-1}\mu^*\vartheta_2^*:\E^n\to\E^n$ which preserves the $\Z$-lattice $L:=\Z e_1+\ldots\Z e_n$. 
And conversely, any linear isomorphism $\varphi:\E_n\to\E_n$ preserving the lattice $L$ yields an isomorphism $\mu:G_1/\sqrt{N_1}\to G_2/\sqrt{N_2}$ with $\varphi=(\vartheta_1^*)^{-1}\mu^*\vartheta_2^*$.
Assume now that $G_1$, $G_2$ are finitely presented. Let $\Sigma_1,\Sigma_2\subseteq S^{n-1}$ so that
$$\Sigma^1(G_i)\cap S(G_i,N_i)=[\vartheta_i^*(\Sigma_i)]$$
for $i=1,2$.
 Using Theorem \ref{twistedfinpres} we deduce that if there is some $\varphi:\E_n\to\E_n$ preserving the lattice $L$ and such that
 $$[\varphi(\Sigma_2^c)]\subseteq-\Sigma_1,$$
then there is some finitely presented normal fibre product $H_\mu$ in $G_1\times G_2$ so that $N_i=H_\mu\cap G_i$ for $i=1,2.$
\end{notation}

The next example illustrates a possible method to construct such a $\mu$ and substitutes a different construction in a previous version of the paper.  This and also the possibility of extending it to higher dimensions were suggested by the referee.

 \begin{example} Consider the following set
 $$\Sigma:=\Omega\cup-\Omega\subseteq S^{1}$$ where $\Omega$ is the closed arc between $[(-2,1)]$ and $[(2,1)]$.
Assume we want to construct a linear map $\varphi:\E^2\to\E^2$ such that $[\varphi(\Sigma^c)]\subseteq-\Sigma$. A possible way to do it is to  \lq\lq move $\E^2$ towards $\R(1,0)$". This is achieved for example by
 the linear map $\varphi_\alpha$ with matrix
 $$\begin{pmatrix}
 1&\alpha\\
 0&1\\
 \end{pmatrix}$$
when $\alpha$ is big enough. More explicitly, $[\varphi_\alpha(\Omega)]$ is the closed arc determined by the intersection of $S^1$ with the region between the half rays $\R^+(\alpha-2,1)$ and  $\R^+(\alpha+2,1)$ thus $[\varphi_\alpha(\Sigma^c)]\subseteq-\Sigma$ if and only if $\alpha>4$. Obviously, the idea behind this is that we are \lq\lq putting more weight in the first coordinate". Note however that one has to do that carefully: consider a linear map $\phi_\alpha$ with matrix of the form
 $$\begin{pmatrix}
 \alpha&1\\
 -1&0\\
 \end{pmatrix}.$$
 This map seems to be moving everything towards $\R(1,0)$ but note that for $v:=(-1,\alpha)$ we have $\phi_\alpha(v)=(1,0)\in\Sigma^c$ and for any $\alpha>0$, $[v]\in\Sigma^c$.

   \end{example}

   \begin{proposition}\label{tech} Let $\Sigma_1,\Sigma_2\subseteq S^{n-1}$ be open sets and $k\leq m$ with $m+k=n$ such that for the subspaces
   $U=\R e_1+\ldots+\R e_k,W=\R e_1+\ldots +\R e_m\leq\E^n$ we have $[U]\subseteq\Sigma_2$, $[W]\subseteq\Sigma_1$.
   Then there is a linear map $\varphi:\E^n\to\E^n$ preserving the lattice $L$ such that
   $$[\varphi(\Sigma_2^c)]\subseteq-\Sigma_1.$$
   \end{proposition}
   \begin{proof} Note first that by openness, there is some $\epsilon>0$ such that 
$$\{v\in S^{n-1}\mid d(v,[U])^2<\epsilon\}\subseteq\Sigma_2,$$
$$\{v\in S^{n-1}\mid d(v,[W])^2<\epsilon\}\subseteq-\Sigma_1.$$
(We are using the obvious fact that $[W]=-[W]\subseteq-\Sigma_1$).
Let $0\leq\alpha\in\Z$ and consider the map $\varphi_\alpha:\E^n\to\E^n$ given by the matrix
$$ \Bigg( \begin{array}{ccccc}
    I_k &\vline &0 &\vline &\alpha I_k \\ 
\hline
    0 &\vline &I_{m-k} &\vline &0 \\ 
\hline
 0 &\vline &0 &\vline & I_k \\ 

  \end{array}\Bigg).$$
Obviously, this is an invertible linear map that preserves the lattice $L$.  Let $T=\R e_{m+1}+\ldots+\R e_n$.
For $v\in S^{n-1}$  put $v=a+b$ with $a\in W$, $b\in T$ (we fix this notation throughout the proof). Then
$\varphi_\alpha(v)=a+f_\alpha(b)+b$
with $f_\alpha:T\to W$ an injective map with image $U$.

\bigskip

\noindent{\bf Step 1}: There is an $M$, not depending on $\alpha$, such that 
$d([\varphi_\alpha(v)],[W])^2>\epsilon$
if and only if
\begin{equation}\label{im}
{\|a+f_\alpha(b)\|^2\over\|b\|^2}<M.
\end{equation}

Observe that $1=\|v\|^2=\|a\|^2+\|b\|$ and $\|\varphi_\alpha(v)\|^2=\|a+f_\alpha(b)\|^2+\|b\|^2$. The point in 
$[W]$ closest to $[\varphi_\alpha(v)]$ is $[a+f_\alpha(b)]$. Then

 $$\begin{aligned}
 &\|[\varphi_\alpha(v)]-[a+f_\alpha(b)]\|^2=\\
&= \|a+f_\alpha(b)\|^2\Bigg|{1\over \|\varphi_\alpha(v)\|}-{1\over\|a+f_\alpha(b)\|}\Bigg|^2+{\|b\|^2\over\|\varphi_\alpha(v)\|^2}=2-2{\|a+f_\alpha(b)\|\over\|\varphi_\alpha(v)\|}.\\
 \end{aligned}$$
 So if we put
 $$\beta={\|a+f_\alpha(b)\|^2\over\|b\|^2}+1,$$
 the claim follows since
 $$d([\varphi_\alpha(v)],[W])^2=2\Big(1-\sqrt{1-{1\over\beta}}\Big).$$

\noindent{\bf Step 2}: There is some $0<\epsilon_1$ such that if $\|b\|^2<\epsilon_1$ and (\ref{im}) holds, then
\begin{equation}\label{v}
d(v,[U])<\epsilon.
\end{equation}
We begin by computing that distance. Let $U':=\R e_{k+1}+\ldots+\R e_m$ and put $a=u+u'$ with $u\in U$, $u'\in U'$. The point of $[U]$ which is closest to $v$ is $[u]$. Then (recall that $\|u\|^2+\|u'\|^2+\|b\|^2=1$)
$$d(v,[U])^2=d(v,[u])^2=\Big\|u-[u]+u'+b\Big\|^2=2(\|u'\|^2+\|b\|^2).$$
Take $\epsilon_1:={\epsilon\over 2(M+1)}.$ As  (\ref{im}) implies 
$\|u'\|^2\leq M\|b\|^2\leq M\epsilon_1$
we have
$$d(v,[U])^2=2(\|u'\|^2+\|b\|^2)\leq 2(M\epsilon_1+\epsilon_1)=\epsilon.$$

\noindent{\bf Step 3}: There is some $\alpha$ big enough such that if  (\ref{im}) holds, then $\|b\|^2<\epsilon_1$. To see it note that 
$$d(v,-f_\alpha(b))^2=\|v+f_\alpha(b)\|^2=\|a+f_\alpha(b)\|^2+\|b\|^2\leq (M+1)\|b\|^2\leq M+1.$$
As $\|f_\alpha(b)\|=\alpha\|b\|$, the three points $0$, $v$ and $-f_\alpha(v)$ form a triangle whose sides have lengths $1$, $\alpha\|b\|$ and $d(v,-f_\alpha(b))$.
By the triangle inequality this means
$$\alpha\|b\|\leq 1+d(v,-f_\alpha(b))\leq 1+\sqrt{M+1}$$
thus
$$\|b\|\leq{1+\sqrt{M+1}\over\alpha}.$$

We can now finish the proof. Let $\alpha$ be as in Step 3. If $v\in S^{n-1}$ is such that $[\varphi_\alpha(v)]\in-\Sigma_1^c$, then 
$v\in\Sigma_2$.  Therefore 
  $$[\varphi(\Sigma_2^c)]\subseteq-\Sigma_1.$$

   \end{proof}

The next result is surely well known, but I have not been able to find an explicit reference in the literature except of the case of co-rank 1 which is \cite{Strebelnotes} Corollary A4.13.

\begin{lemma}\label{rational} Let $G$ be finitely generated and $G'\leq N\leq G$. The following conditions are equivalent:
\begin{itemize}
\item[i)] There is a finitely generated subgroup $G'\leq N\leq K\leq G$ of co-rank $k$,

\item[ii)] there is a rationally defined linear subspace $V\leq\E^n$ of dimension $k$ with $[\vartheta(V)]\subseteq\Sigma^1(G)\cap S(G,N)$,

\item[iii)] there is a linear subspace $V\leq\E^n$ of dimension $k$ with $[\vartheta(V)]\subseteq\Sigma^1(G)\cap S(G,N)$.
\end{itemize}
\end{lemma}
\begin{proof}  Let $G'\leq N\leq G$ be a normal subgroup of co-rank $n$. Assume we have a $k$-dimensional subspace $\text{Hom}(G/N,\R)$ generated by discrete characters $\chi_1\ldots,\chi_k:G/N\to\R$. Then for the subgroup $K=\text{Ker}\chi_{1}\cap\ldots\cap\text{Ker}\chi_{k}$ we have
$N\leq K\leq G$ and $\text{rk}G/K=k$. Conversely, for any $N\leq K\leq G$ of co-rank $k$, the subspace generated in $\text{Hom}(G/N,\R)$ by the subsphere $S(G,K)$ has dimension $k$ and is generated by discrete characters. (If we start with $K$, consider the associated subspace $S(G,K)$ and then the intersection of the kernels of a basis, what we recover is $\sqrt{K}$).  Moreover, Theorem \ref{bierirenz} implies that the subgroup $K$ is finitely generated if and only if $S(G,K)$ lies in $\Sigma^1(G)\cap S(G,N).$ 

Note also that under the coordinate map $\vartheta$, any linear subspace $V$ of $\E^n$ corresponds to a linear subspace of $\text{Hom}(G/N,\R)$ so that
 $V$ is rationally defined if and only if $\vartheta(V)$ is generated by discrete characters. 
This yields the equivalence between i) and ii). To see that iii) implies ii) one can proceed as follows: 
Choose a basis $\{v_1,\ldots, v_k\}$ of $V$ so that for $i=1,\ldots,k$, $\|v_i\|=1$ and the $v_i$ are pairwise orthogonal. For some $\epsilon>0$ small enough, choose $u_1,\ldots,u_k\in\E^n$ so that  for every $i=1,\ldots,k$ we have $\|u_i\|=1$, $\|u_i-v_i\|<\epsilon$, there is some rational vector in $\R u_i$ and moreover the set $\{u_1,\ldots,u_k\}$ is linearly independent. The fact that there is a family with the first properties follows from the density result proven in \cite{Strebelnotes} Lemma B3.3, it is possible to check that the $u_i$'s can be taken to be linear independent by proceeding by induction, using the fact that each $v_i$ is \lq\lq far away" from $S^{n-1}\cap(\R v_1+\ldots+\R v_{i-1})$. Alternatively, take into account that once $u_1,\ldots,u_{i-1}$ have been chosen then the set of those $p\in S^{n-1}$ such that $\|p-v_i\|<\epsilon$ and $p\not\in\R u_1+\ldots+\R u_{i-1}$ is open in $S^{n-1}$ so the claim follows from the mentioned density result. Then one easily checks that making $\epsilon$ small, $[U]$ and $[V]$ are \lq\lq close enough" so that if $[\vartheta(V)]\subseteq\Sigma^1(G)$, then $[\vartheta(U)]\subseteq\Sigma^1(G)$ (recall that $\Sigma^1(G)$ is open).

\end{proof}

Note that in all the examples below we will use the equivalence between i) and ii) only. 

\begin{example} For most groups $G$ for which is known, the invariant $\Sigma^1(G)^c$ is rationally defined, more explicitly, it is a union of intersections of rationally defined closed hemispheres.  However, this is not always true.
Let $F$ be Thompson's group, i.e., the group of dyadic piecewise linear self-homeomorphisms of the unit interval having finitely many break points. Then $F/F'$ has torsion free rank 2 and $\Sigma^1(F)^c$ consist of two points, but they are not discrete, in other words, the associated subset in $S^1$ is not rationally defined (see \cite{BieriNeumannStrebel}). 
\end{example}

\begin{theorem}\label{genasym} Let $G_1,G_2$ be finitely presented, $G_1'\leq N_1\leq G_1$, $G_2'\leq N_2\leq G_2$ such that $G_1/N_1\cong G_2/N_2$. Assume that there are  $N_1\leq K_1\leq G_1$, $N_2\leq K_2\leq G_2$ both finitely generated and of co-rank $m$ and $k$ respectively such that 
$$k+m=n:=\text{rk}G_1/N_1=\text{rk}G_2/N_2.$$
 Then there is some normal finitely presented fibre product $H$ in $G_1\times G_2$ with $H\cap G_i=N_i$ for $i=1,2.$
\end{theorem}
\begin{proof} We may assume that that $k\leq m$. If $m=n$ it suffices to consider the fibre product $H_\mu$ associated to any $\mu:G_1/N_1\to G_2/N_2$ and use either  the asymmetric 1-2-3 Theorem, or Theorem \ref{twistedfinpres}. Use Notation \ref{coordinates} and observe that by choosing suitable generating systems of $G_1,G_2$, we may assume that
$[\vartheta_2(U)]=S(G_2,K_2)$ and $[\vartheta_1(W)]=S(G_1,K_1)$
for  $U=\R e_1+\ldots+\R e_k$ and $W=\R e_1+\dots+\R e_m$. Using Proposition \ref{tech} we deduce that there is some $\varphi:\E^n\to\E^n$ preserving the lattice $L$ with $[\varphi(\Sigma_2)]\subseteq-\Sigma_1$. By the discussion at the end of Notation \ref{coordinates} this implies the existence of an isomorphism $\hat\mu:G_1/\sqrt{N_1}\to G_2/\sqrt{N_2}$ so that $H_{\hat\mu}$ is finitely presented. Then $\hat\mu$ can be lifted to an iso $\mu:G_1/N_1\to G_1/N_2$ so that $H_\mu$ and $H_{\hat\mu}$ are commensurable so we get the result.
\end{proof}

\begin{corollary}\label{linear}
Let $G_1,G_2$ be finitely presented and put $n_i:=\text{rk}G_i/G_i'$. 
Assume that there are  $G_1'\leq K_1\leq G_1$, $G_2'\leq K_2\leq G_2$ both finitely generated and of co-rank $m$ and $k$ respectively. Then there is some normal finitely presented fibre product in $G_1\times G_2$ of co-rank 
$$n:=\text{min}\{m+k,n_1,n_2\}.$$ 
\end{corollary}
\begin{proof} Let $n:=\text{min}\{m+k,n_1,n_2\}$. There are subgroups $G_i'\leq N_i\leq G_i$ such that $n=\text{rk}G_1/N_1=\text{rk}G_2/N_2$ and $\text{rk}G_1/N_1K_1+\text{rk}G_2/N_2K_2\geq n$. As the subgroups $N_1K_1$ and $N_2K_2$ are finitely generated, this means that working in the sub spheres $S(G_1,N_1)$ and $S(G_2,N_2)$ and taking suitable overgroups of $N_1K_1$ and $N_2K_2$ we may assume that $n=k+m$. Now, use Theorem \ref{genasym}.
 \end{proof}
 
\begin{remark} The asymmetric 1-2-3 theorem  implies the same but for co-rank 
$$\text{min}\{m,k\}.$$
\end{remark}

In some cases it is not difficult to detect the existence of a finitely generated subgroup of a given co-rank as  in Theorem \ref{genasym} or, equivalently, of a linear subspace $U$ with $[\vartheta(U)]\subseteq\Sigma^1(G)$. 
 
\begin{example}\label{2tame} Let $G$ be a finitely presented subgroup that contains no non abelian free subgroup. Then by \cite{BieriNeumannStrebel} Theorem D (the explicit argument will be recalled in Theorem \ref{rank2} below) there is a  finitely generated subgroup $G'\leq K<G$ of co-rank 1.  \end{example}

\begin{example}
If  $G$ is the fundamental group of a compact 3-manifold, \cite{BieriNeumannStrebel} Corollary F implies that $\Sigma^1(G)=-\Sigma^1(G)$. Therefore, if $\Sigma^1(G)\neq\emptyset,$ there is some $[\chi]$ with $[\chi],[-\chi]\in\Sigma^1(G)$. Thus again by Lemma \ref{rational} we deduce that there is a co-rank 1 finitely generated subgroup $G'\leq K<G$.
\end{example}

The next example was pointed out by the referee:

\begin{example}\label{product} Assume that $G=L_1\times L_2$ with $L_1,L_2$ finitely presented. Let $n=\text{min}\{\text{rk}L_1/L_1',\text{rk}L_2/L_2'\}$. There is some normal fibre product $K$ in $G$ of co-rank $n$ and by \cite{BaumslagRoseblade} Lemma 2, $K$ is finitely generated.

\end{example}

\begin{example}\label{prefer} Let $G$ be virtually solvable and finitely presented and assume that $G$ has finite Pr\"ufer rank (i.e., there is a bound on the smallest number of generators of every finitely generated subgroup of $G$). Let $n=\text{rk}G/G'$. Then there is a finitely generated $G'\leq K\leq G$ of co-rank $n-1$. To see it recall that by \cite{meinert2} Theorem 2.5 (see also \cite{BieriStrebel1} Proposition 2.5),  $\Sigma^1(G)^c$ is finite (and consists of discrete characters only). Then taking any hyperplane $U\subseteq\E^n$ such that $\Sigma^1(G)^c\cap\vartheta(U)=\emptyset$ and using Lemma \ref{rational} we get the claim. 
\end{example}

\begin{corollary} Assume that $G_1,G_2$ are finitely presented, that $G_2$ is virtually solvable of finite Pr\"ufer rank and that $G_1$ contains no non abelian free subgroup. Then there is some normal finitely presented fibre product in $G_1\times G_2$ of maximal co-rank (i.e. of co-rank $\text{min}\{\text{rk}G_1/G_1',\text{rk}G_2/G_2'\}$). 
\end{corollary}

In the particular case when $G=G_1=G_2$ we get:

\begin{corollary}\label{doublerank} Let $G$ be finitely presented and assume that there is a finitely generated $H$ with $N\leq H\leq G$ such that $\text{rk}G/H\geq 2\text{rk}G/N$. Then there is some finitely presented twisted $N$-fibre product.
\end{corollary}

\begin{example}\label{1relatorposword} Let $G$ be a 2-generated 1-relator group with relator of the form $u=v$ where $u$, $v$ are positive words on the generator so that the sum of the exponents of each of the generators in $u$ equals that in $v$.  Then there is some finitely presented twisted $G'$-fibre product in $G\times G$. To see it, note that $\text{rk}G/G'=2$ and recall that by \cite{Baumslag} Theorem 1, $G$ has a normal subgroup $N$ which is finitely generated free and such that $G/N$ is cyclic so we only have to apply Corollary \ref{doublerank}.  
\end{example}

\section{Some partial converses}

In this section we are going to assume that $G=G_1=G_2$ and consider situations when the converse of Corollary \ref{doublerank} holds true.

\begin{theorem}\label{rank2} Let $G$ be finitely presented and $G'\leq N\leq G$ of co-rank 2. Then there is a finitely presented normal $N$-fibre product  in $G\times G$ if and only if
$G=<t>\ltimes K$ where $K$ is finitely generated and $N\leq K$.
\end{theorem}
\begin{proof} For the \lq\lq if" direction it suffices to use Corollary \ref{doublerank}. In fact, note that
the condition $G=<t>\ltimes K$ with $K$ finitely generated is equivalent to the existence of a discrete character $\chi:G\to\R$ so that $[\chi],[-\chi]\in\Sigma^1(G)^c$ (use Lemma \ref{rational} or \cite{Strebelnotes} Corollary A4.3).

Assume that there  is a finitely presented $N$-fibre product $H\leq G\times G$. By Theorem \ref{twistedfinpres}, then there is some $\mu:G/N\to G/N$ such that
$$[\mu^*(\Sigma^1(G)^c)]\cap S(G,N)\subseteq-\Sigma^1(G)\cap S(G,N).$$
 In the case when $\Sigma^1(G)^c\cap S(G,N)$ is not 2-tame, there are $[\nu],[-\nu]\in\Sigma^1(G)^c\cap S(G,N)$ thus for $\chi:=\varphi(\nu)$, $[\chi],[-\chi]\in\Sigma^1(G)\cap S(G,N)$. So we are left with the case when $\Sigma^1(G)^c\cap S(G,N)$ is 2-tame. The argument of \cite{BieriNeumannStrebel} Theorem D, which we recall now finishes the proof. Take $[\nu]$ in the boundary of $\Sigma^1(G)^c\cap S(G,N)$, then $[-\nu]\in\Sigma^1(G)$ and by openness there is some $[\chi]$ close to $[\nu]$ so that $[\chi],[-\chi]\in\Sigma^1(G)\cap S(G,N)$.
\end{proof}

\begin{theorem}\label{greatsph} Let $G$ be a group with a finite family $\Lambda$ of subgroups $G'\leq K\leq G$ such that 
$$\Sigma^1(G)^c=\bigcup_{K\in\Lambda} S(G,K).$$
There is a finitely presented $G'$-fibre product in $G\times G$ if and only if  $\text{rk}G/G'\leq 2\text{rk}K/K'$ for any $K\in\Lambda$.
\end{theorem}
\begin{proof} Assume first that the condition on the ranks of the elements in $\Lambda$ holds and let $n=\text{rk}G/G'$. Let $V$ be a subspace of $\E^n$ of biggest possible dimension with zero intersection with  the vector space generated by each of the sub spheres $S(G,K)$ (such a subspace exists, see for example \cite{Strebelnotes} Lemma A4.9). 
Then $[V]\cap S(G,K)=0$ for any $K\in\Lambda$ thus $[V]\subseteq\Sigma^1(G)$ and for some $K\in\Lambda$, $\text{dim}_\R V+\text{rk}G/K=n$ which implies that $n\leq 2\text{dim}_\R V$ so from Corollary \ref{doublerank} we deduce that there is some finitely presented $G'$-fibre product in $G\times G$.

Conversely, if there is such a fibre product then by Corollary \ref{untwistedfinpres} we see that there is some $\varphi:\E^n\to\E^n$ with $[\varphi(S(G,K))]\cap S(G,K)=\emptyset$ for any $K\in\Lambda$. This implies that $2n-2\text{rk}K/K'=2\text{rk}G/K\leq n$.
\end{proof}

\begin{example} The hypothesis of Theorem \ref{greatsph} hold for pure symmetric automorphisms of finitely generated free groups (\cite{Orlandi}) and  fundamental groups of compact K\"ahler manifolds (\cite{Delzant}). It is also the case for right-angled Artin groups. For this last family we are going to be more explicit: let $\Delta$ be a flag complex and $G_\Delta$ the associated right angled Artin group. If $S\subseteq V(\Delta)$ is a subset of vertices we denote by $\Delta_S$ the smallest subcomplex of $\Delta$ containing $S$ and set $G_S:=G_{\Delta_S}$
 seen as a subgroup of $G$.
 By \cite{Strebelnotes} Proposition A4.14 (due to Meier and VanWyk) if $G$ is not abelian,
\begin{equation}\label{raag}\Sigma^1(G)^c=\bigcup_{S\in\mathcal{S}}S(G,G_S)\end{equation}
where $\mathcal{S}$ is the set of subsets $S\subseteq V(\Delta)$ such that the subcomplex of $\Delta$ obtained by removing the vertices in $S$ is disconnected and $S$ is minimal with respect to that property.

\end{example}

So we get:

\begin{corollary}\label{artin} Let $G:=G_\Delta$ be a right-angled-Artin group. Then
\begin{itemize}
\item[i)] The untwisted $G'$-fibre product of $G\times G$ is finitely presented if and only if $G$ is abelian.

\item[ii)] Let $n=|V(\Delta)|$. Then there is some twisted finitely presented $G'$-fibre product in $G\times G$ if and only if $n\leq 2|S|$ for any $S\in\mathcal{S}$.
\end{itemize}
\end{corollary}

\begin{example} Let $\Delta$ be the graph

 \bigskip

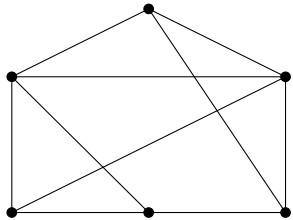
\begin{tikzpicture}[scale=0.9]
\filldraw(1,-2) ;
  \draw[black] (5,-2) -- (7,-2) --(9,-2)--(9,0)--(7,1)--(5,0)--(9,0)--(5,-2)--(5,0)--(7,-2);
   \draw[black] (7,1)--(9,-2); 
  \filldraw (5,0) circle (2pt) ;
   \filldraw (7,1) circle (2pt);
   \filldraw (9,0) circle (2pt);
       \filldraw (5,-2) circle (2pt);
    \filldraw (7,-2) circle (2pt) ;
   \filldraw (9,-2) circle (2pt) ;

   \end{tikzpicture}

\bigskip

\noindent One easily checks there is no set of 2 vertices such that its removal disconnects the graph so the condition of the Theorem is satisfied. However, $V(\Delta)$ can not be split as a disjoint union of two sets of vertices so that each of the vertices in the first set commutes with each of the vertices of the second set. This means that the associated right angled Artin group is not a direct product thus it does not fall into the scope of Example \ref{product}.
\end{example}

\end{document}